\documentclass{amsart}

\usepackage{amsmath}
\usepackage{amsthm}
\usepackage{amssymb}
\usepackage{hyperref}

\theoremstyle{definition} \newtheorem{definition}{Definition}[section]

\theoremstyle{plain} \newtheorem{theorem}[definition]{Theorem}
\theoremstyle{plain} 
\theoremstyle{plain} \newtheorem{lemma}[definition]{Lemma}
\theoremstyle{plain} \newtheorem{proposition}[definition]{Proposition}
\theoremstyle{plain} \newtheorem{corollary}[definition]{Corollary}
\theoremstyle{plain} 
\theoremstyle{plain} 
\theoremstyle{plain} 
\theoremstyle{plain} 

\theoremstyle{remark} \newtheorem{remark}[definition]{Remark}

\newcommand{\mcC}{\mathcal{C}}

\newcommand{\mcL}{\mathcal{L}}
\newcommand{\mcP}{\mathcal{P}}

\newcommand{\wto}{{\widetilde O}}

\newcommand{\eps}{\varepsilon}

\newcommand{\ceil}[1]{\ensuremath{\left\lceil #1 \right\rceil}}

\newcommand{\mfX}{\mathfrak{X}}

\newcommand{\sr}{strongly regular}

\begin{document}

\title{Asymptotic Delsarte cliques in distance-regular graphs}

\author{L\'aszl\'o Babai}
\address{Departments of Computer Science and Mathematics \\ University of Chicago}
\email{laci@cs.uchicago.edu}

\author{John Wilmes}
\address{Department of Mathematics \\ University of Chicago}
\email{wilmesj@math.uchicago.edu}

\thanks{The 
first author was supported in part by NSF Grants CCF-1017781
and CCF-7443327. The
second author was supported in part by NSF Grant DGE-1144082.  The statements
made in the paper are those of the authors and do not necessarily reflect
the views of the NSF.}

\date{20 November 2014}

\maketitle

\begin{abstract}
We give a new bound on the parameter $\lambda$ (number of common 
neighbors of a pair of adjacent vertices) in a distance-regular
graph $G$, improving and generalizing bounds for strongly regular graphs
by Spielman (1996) and Pyber (2014). The new bound is one of the ingredients of
recent progress on the complexity of testing isomorphism of strongly regular
graphs (Babai, Chen, Sun, Teng, Wilmes 2013).  The proof is based on a
clique geometry found by Metsch (1991) under certain constraints on the
parameters.  We also give a simplified
proof of the following asymptotic consequence of Metsch's result: if 
$k\mu = o(\lambda^2)$ 
then each edge of $G$ belongs to a unique maximal clique of size
asymptotically equal to $\lambda$, 
and all other cliques have size $o(\lambda)$.  Here $k$ denotes
the degree and $\mu$ the number of common neighbors of a pair of vertices at
distance 2.    We point out that Metsch's cliques are ``asymptotically
Delsarte'' when $k\mu = o(\lambda^2)$, so families of distance-regular
graphs with parameters satisfying $k\mu = o(\lambda^2)$
are ``asymptotically Delsarte-geometric.'' 
\end{abstract}

\maketitle

\section{Introduction}

A graph is called \emph{amply regular} with parameters $(n,k,\lambda,\mu)$ if
it is $k$-regular on $n$ vertices, any two adjacent vertices have exactly
$\lambda$ common neighbors, and any two vertices at distance two from each
other have exactly $\mu$ common neighbors. Amply regular graphs have been
well-studied, as they generalize distance-regular graphs while preserving many
of their properties~\cite[Section 1.1]{bcn-drg}. Our first result gives a new
bound on $\lambda$.

In fact, our bound applies more generally to ``sub-amply regular'' graphs. We
say a graph is \emph{sub-amply regular} when it satisfies the weaker condition
that any two vertices at distance two from each other have \emph{at most} $\mu$
common neighbors.

\begin{theorem}\label{thm:lambda}
Let $G$ be a sub-amply regular graph with parameters 
$(n,k,\lambda,\mu)$ which is
not a disjoint union of cliques. Then
\begin{equation}\label{eq:lambda-explicit}
\lambda +1 < 
  \max\left\{4\sqrt{2n}, \frac{6}{\sqrt{13}-1}\sqrt{k(\mu-1)}\right\}.
\end{equation}
\end{theorem}
Even in the very special case of strongly regular graphs, this result 
considerably improves the previously known bounds for $\lambda$
(Spielman~\cite{spielman} and Pyber~\cite{pyber}) 
in some ranges of the parameters.  (See Sec.~\ref{sec:srg}) for
a detailed comparison.)  The new bound was used in~\cite{bcstw-sr} to
improve Spielman's $\exp(\wto(n^{1/3}))$ bound on the complexity
of testing isomorphism of \sr\ graphs to $\exp(\wto(n^{1/5}))$ where 
the $\wto$ notation hides polylogarithmic ($(\log n)^C$) factors.  
This application was the key motivation of the present paper 
(see Sec.~\ref{sec:iso}). 

\subsection{Clique geometry}
We say that a collection $\mcC$ of cliques of a graph is 
a \emph{clique geometry} if (i) all cliques in $\mcC$
are maximal and (ii) every pair of adjacent vertices
of $G$ belongs to a unique member of $\mcC$.
We shall refer to the members of $\mcC$ as \emph{special cliques}.

Our result relies on the remarkable clique geometry appearing in
sub-amply regular graphs under certain constraints on the
parameters, discovered by Metsch~\cite{metsch1} (Theorem~\ref{thm:metsch}).
We observe in particular that Metsch's constraints are met when 
$k\mu/\lambda^2$ is small; furthermore, in this case, the
special cliques have nearly uniform order.
(The \emph{order} of a clique is the number of its vertices.)

Sub-amply regular graphs $G$ with $\mu\le 1$ trivially have a (unique)
clique geometry. 
When $\mu=0$, $G$ is a disjoint union of cliques of order $\lambda
+ 2 = 1 + k$. When $\mu=1$, the common neighbors of two adjacent
vertices form a clique. When $\mu=2$ and $k < (1/2)\lambda(\lambda+3)$, Brouwer
and Neumaier showed that again the common neighbors of any pair of adjacent
vertices form a clique~\cite{brouwer-neumaier}. In such graphs, every edge lies
in a unique maximal clique, and every maximal clique has order exactly $\lambda
+ 2$.  
A clique geometry exists under much more general conditions, as proved by
Metsch~\cite{metsch1}.

\begin{theorem}[{Metsch~\cite[Result 2.1]{metsch2}}]\label{thm:metsch}
Let $G$ be a sub-amply regular graph with parameters $(n,k,\lambda,\mu)$, and
let $t$ be an integer such that
\begin{align*}
    \lambda &> (2t-1)(\mu-1) - 1, \textrm{ and }\\
    k &< (t+1)(\lambda + 1) - \frac{1}{2}t(t+1)(\mu-1).
\end{align*}
Then the maximal cliques of order at least $\lambda + 2 - (t-1)(\mu-1)$
form a clique geometry, and
each vertex belongs to at most $t$ special cliques.
\end{theorem}

\begin{remark}We note that special cliques of Theorem~\ref{thm:metsch} can be
    easily recognized by the degree of the vertices in the common neighborhood
    of a pair of adjacent vertices. In particular, if $u$ and $v$ are two
    adjacent vertices of $G$, then a common neighbor $w$ of $u$ and $v$ lies
    in the special clique containing $u$ and $v$ iff in the subgraph of $G$
    induced on the common neighborhood of $u$ and $v$, the degree of $w$ is at
    least $\lambda - (t-1)(\mu - 1) - 1$.
\end{remark}

\begin{corollary}\label{cor:metsch}
Let $G$ be a sub-amply regular graph with parameters $(n,k,\lambda,\mu)$ such
that
\begin{equation}\label{eq:main}
    (\lambda+1)^2 > (3k + \lambda + 1)(\mu-1).
\end{equation}
Then the maximal cliques of order at least
$\lambda + 2 - (\ceil{(3/2)k/(\lambda+1)} - 1)(\mu-1)$
form a clique geometry.
\end{corollary}

The corollary is obtained from Theorem~\ref{thm:metsch} by setting 
$t = \ceil{3k/(2(\lambda+1))}$.\qed

The starting point of our work was Spielman's 1996 paper~\cite{spielman} in
which he derived asymptotic consequences of Neumaier's 1979 classification of
strongly regular graphs~\cite{neumaier79}, including a bound on the parameter
$\lambda$.  Our bound~\eqref{eq:lambda-explicit} applies more generally to
sub-amply regular graphs (and hence does not require Neumaier's
classification), and improves Spielman's bound for $k > n^{5/8}$.

We prove the bound~\eqref{eq:lambda-explicit} in Section~\ref{sec:lambda}. In
Section~\ref{sec:srg}, we compare Spielman's bound and Pyber's bound to our
own. Then, in Section~\ref{sec:iso}, we explain the connection to graph
isomorphism testing in some detail. 

The asymptotic viewpoint makes the results considerably more
transparent.
In Section~\ref{sec:clique}, we give a short self-contained proof of 
Theorem~\ref{thm:clique-asymptotic} (below), an
asymptotic corollary to Metsch's theorem.

To interpret asymptotic statements such as ``Let $G$ be an amply regular graph
with $k\mu = o(\lambda^2)$,'' we think of our graph $G$ as belonging to some
infinite family for which the asymptotic relation holds.  All hidden constants
are absolute, and all limits are uniform as the number of vertices $n \to
\infty$.  We use 
common notation for asymptotic relations,
including writing $f \sim g$ (asymptotic equality) for functions $f$ and $g$
for which $\lim_{n\to\infty}(f(n)/g(n)) = 1$.  We write $f(n)\gtrsim g(n)$ if
$f(n) \sim \max\{f(n),g(n)\}$. 

\begin{theorem}\label{thm:clique-asymptotic}
Let $G$ be a sub-amply regular graph with parameters $(n,k,\lambda,\mu)$ such
that $k\mu = o(\lambda^2)$.  Then (for $n$ sufficiently large) every pair of
adjacent vertices belongs to a unique maximal clique of order $\sim \lambda$,
and all other maximal cliques in $G$ have order $o(\lambda)$.
\end{theorem}
So the large maximal cliques form a clique geometry.

The key lemma used in the proof, Lemma~\ref{lem:clique-local}, is used in the
recent characterization of primitive coherent configurations with more than
$\exp(n^{1/3 + \eps})$ automorphisms by Sun and Wilmes~\cite{sun-wilmes}.

\subsection{Asymptotic Delsarte geometry}
Let $s$ denote the least eigenvalue of (the adjacency matrix of) 
the graph $G$ (so $s < 0$).
The following bound on the order of cliques in
distance-regular graphs was established by Delsarte.
\begin{lemma}[Delsarte~\cite{delsarte}]\label{lem:delsarte-bound}
   If $G$ is a distance-regular graph then 
   no clique in $G$ has order greater than $1+k/|s|$.
\end{lemma}
Any clique achieving this order is called a \emph{Delsarte
clique}~\cite{godsil-ac}.

We call a graph $G$ \emph{Delsarte-geometric} 
if $G$ is distance-regular and it has a clique geometry in which
all special cliques are Delsarte.
This concept was introduced by Godsil~\cite{godsil-geom} who called
such graphs ``geometric.''
Johnson and Hamming graphs are examples of Delsarte-geometric graphs.  

Godsil~\cite{godsil-geom} gave the
following sufficient condition for a distance-regular
graph to be Delsarte-geometric.

An \emph{$m$-claw} in a graph is an induced $K_{1,m}$ subgraph.

\begin{theorem}[Godsil~{\cite{godsil-geom}}] Let $G$ be a distance-regular graph with
least eigenvalue $s$. If there are no $m$-claws in $G$ with $m > |s|$ and
\begin{equation}\label{eq:godsil}
    \lambda + 1 > (2|s|-1)(\mu-1) 
\end{equation}
then $G$ is Delsarte-geometric.
\end{theorem}

It would seem desirable to replace the structural assumption (bound on
claw size) in Godsil's theorem by a reasonable assumption involving
the parameters of the graph only since this would allow broader
applicability of the result. Bang and Koolen make a step in this direction,
removing the structural assumption but strengthening the 
constraint on the parameters.
\begin{theorem}[Bang, Koolen~\cite{bang-koolen}]
If $\lambda > \lfloor s\rfloor ^2 \mu$ 
for a distance-regular graph $G$ with least
eigenvalue $s$ 
then $G$ is Delsarte-geometric.
\end{theorem}

Note that for large $|s|$, the Bang--Koolen constraint
$s^2 \mu \lesssim \lambda$
requires essentially a factor 
of $|s|/2$ larger $\lambda$ than does Godsil's 
constraint~\eqref{eq:godsil} which for large $|s|$
and $\mu$ requires $2|s|\mu \lesssim \lambda$.

We point out that already an increase by a factor that goes
to infinity arbitrarily slowly
compared to Godsil's contraint, $|s|\mu = o(\lambda)$,
suffices for an \emph{asymptotic} Delsarte geometry,
i.\,e., a clique geometry where the order of the special
cliques is $\sim k/|s|$.  

\begin{theorem}\label{thm:delsarte}
Let $G$ be a distance-regular graph satisfying
$|s|\mu = o(\lambda)$.
Then $G$ is asymptotically Delsarte-geometric.
\end{theorem}

Theorem~\ref{thm:delsarte} is proved in Section~\ref{sec:delsarte}.

\subsection*{Acknowledgements}
The authors wish to acknowledge the inspiration from their joint
work with Xi Chen, Xiaorui Sun, and Shang-Hua Teng on the
isomorphism problem for strongly regular graphs.

\section{Bounding $\lambda$ in sub-amply regular graphs}\label{sec:lambda}
In this section, we derive our bound on $\lambda$ (Theorem~\ref{thm:lambda})
from Corollary~\ref{cor:metsch}.  

\begin{lemma}\label{lem:hypergraph}
Let $\mcC$ be a geometric collection of cliques in a graph $G$ on $n$ vertices
such that every vertex is in at least $r \ge 2$ and at most $R$ cliques, and
each clique has order at least $\ell$.  Then 
\[
    \ell \le \frac{R}{\sqrt{r(r-1)}}\sqrt{n}.
\]
\end{lemma}
\begin{proof}
Let $m=|\mcC|$ and let $N$ be the number of vertex--clique incidences.
Then $\ell m \le N \le nR$.  Let $T$ be the number of triples
$(v,C_1,C_2)$ where $C_1,C_2\in \mcC$ and $v \in C_1 \cap C_2$.
Then $T = \sum_{v \in V} \deg(v)(\deg(v)-1) \ge nr(r-1)$ (where $V$ is the set
of vertices).  On the other hand, by the intersection assumption, $T\le m(m-1)
< m^2$.  Comparing,
\[
    nr(r-1) < m^2 \le \left(\frac{nR}{\ell}\right)^2\,. \hfill\qedhere
\]
\end{proof}

\begin{proof}[Proof of Theorem~\ref{thm:lambda}]
\emph{Case 1.} Suppose $(3k+\lambda+1)(\mu-1) < (\lambda+1)^2$. Then by
Corollary~\ref{cor:metsch}, every edge lies in a special clique of order at
least $\ell := \lambda + 2 - (3/2)k(\mu-1)/(\lambda+1) > (1/2)(\lambda+1)$. The
number of special cliques containing a given vertex is at most $R :=
2k/(\lambda + 1)$, and at least $k/(\lambda + 1)$. Let $r = \lceil k/(\lambda +
1)\rceil$. Since $G$ is not a disjoint union of cliques, $\lambda + 1 < k$, so
$r \ge 2$. Applying Lemma~\ref{lem:hypergraph} gives $\ell \le
(R/\sqrt{r(r-1)})\sqrt{n} \le (R/r)\sqrt{2n}$. Hence, $\lambda+1 < 4\sqrt{2n}$.

\emph{Case 2.} Otherwise, $(\lambda+1)^2\le (3k+\lambda+1)(\mu-1)$. Set $\delta
=  (\sqrt{13}-1)/6$.

\emph{Case 2a.} Suppose $\mu-1 \ge \delta (\lambda+1)$. Then 
\begin{equation}  \label{eq:case2a}
\lambda +1 \le
(1/\delta)(\mu-1) < (1/\delta)\sqrt{k(\mu-1)}.
\end{equation}

\emph{Case 2b.} Otherwise, $\mu-1 < \delta(\lambda+1)$, and we have
\[
    (1-\delta)(\lambda+1)^2 < 3k(\mu-1),
\]
which is equivalent to Eq.~\eqref{eq:case2a} by our choice of $\delta$.
The Theorem follows by combining Eq.~\eqref{eq:case2a} with Case 1.
\end{proof}

\section{Proof of clique structure} \label{sec:clique}
We now give a simple proof of Theorem~\ref{thm:clique-asymptotic}.  
The core of the proof is the Clique Partition Lemma~\ref{lem:clique-local}
below; the lemma is a consequence of Metsch's~\cite[Theorem 1.2]{metsch1}.
The simplification results from our use of the
following lemma, implicit in Spielman~\cite[Lemma 17]{spielman}.

If $u$ is a vertex of a graph $G$, we write $N(u)$ for the neighborhood of $u$,
i.\,e., the set of vertices adjacent to $u$, and write $N^+(u) =
N(u)\cup\{u\}$.
\begin{lemma}[Spielman]\label{lem:spielman}
    Let $G$ be a graph on $k$ vertices which is regular of degree $\lambda$ and
    such that any pair of nonadjacent vertices has at most $\mu-1$ common
    neighbors. Then for any vertex $u$, there are at most
    $(k-\lambda-1)(\mu-1)$ ordered pairs of nonadjacent vertices in $N(u)$.
\end{lemma}
\begin{proof}
Let $X$ be the number of ordered pairs of nonadjacent vertices in $N(u)$, and
let $K$ be the number of ordered pairs of adjacent vertices in $N(u)$,
so $K + X = \lambda(\lambda-1)$. Let $P$ be the number of ordered pairs $(x,y)$
of vertices such that $(u,x,y)$ induces a path of length two ($u$ and $y$ are
not adjacent, and $x$ is adjacent to both). For every
neighbor $x$ of $u$, and every neighbor $y \ne u$ of $x$, the pair $(x,y)$ is
counted in either $K$ or $P$, so $K + P = \lambda(\lambda-1)$ and so $P = X$.
On the other hand, there are $k-\lambda-1$ vertices not adjacent to $u$, each
of which has at most $\mu-1$ common neighbors with $u$, and so $X = P \le
(k-\lambda-1)(\mu-1)$.
\end{proof}

\begin{lemma}[Clique Partition Lemma (Metsch)] \label{lem:clique-local}
    Let $G$ be a graph on $k$ vertices which is regular of degree $\lambda$
    and such that any pair of nonadjacent vertices has at most $\mu-1$ common
    neighbors. Suppose that $k\mu = o(\lambda^2)$. Then there is a
    partition of $V(G)$ into maximal cliques of order $\sim \lambda$, and all
    other maximal cliques of $G$ have order $o(\lambda)$.
\end{lemma}
\begin{proof}
Fix a vertex $u$ and consider the induced subgraph $H$ of $G$ on
$N^+(u)$. Suppose $x$ and $y$ are distinct non-adjacent
vertices of $H$. They have at most $\mu-1$ common neighbors in $H$, 
so there are at least $\lambda - \mu$ vertices in $H\setminus\{x,y\}$ which are
not common neighbors of $x$ and $y$. Hence, at least one of $x$ and $y$ has
codegree at least $\kappa := (\lambda - \mu)/2$ in $H$ (i.\,e., degree at most
$\lambda - \kappa$). Let $D$ be the set of vertices in $H$ of codegree at least
$\kappa$, and let $C = H\setminus D$. It follows that $C$ is a clique, and
clearly $u \in C$.

Now by Lemma~\ref{lem:spielman}, $|D|\kappa < (k-\lambda-1)(\mu-1) =
o(\lambda^2)$, and so $|D| = o(\lambda)$. In particular, $C \sim \lambda$, and
every element of $D$ has at least one non-neighbor in $C$. Hence, $C$ is a
maximal clique, and every element not in $C$, having at least one non-neighbor
in $C$, has at most $\mu$ neighbors in $C$. Thus, any maximal clique which
contains $u$ as well as a vertex not in $C$ has order at most $|D| + \mu =
o(\lambda)$.  
\end{proof} 

Theorem~\ref{thm:clique-asymptotic} then follows immediately by applying
Lemma~\ref{lem:clique-local} to the graphs induced by $G$ on $N(u)$ for $u \in
V$. \qed

\section{Asymptotic Delsarte cliques}\label{sec:delsarte}

We finally give a prove of Theorem~\ref{thm:delsarte}.

Suppose that $G$ is distance-regular with intersection numbers $b_i, c_i$,
where for any pair $u,v$ of vertices $u$ at distance $i$, the number of
neighbors of $u$ at distance $i+1$ from $v$ is $b_i$ and the number of
neighbors of $u$ at distance $i-1$ from $v$ is $c_i$ (cf. \cite[Chap.
4.1]{bcn-drg}).  (Note that every distance-regular graph is sub-amply regular
with parameters $\lambda = b_0 - b_1 - 1$ and $\mu = c_2$.) 

\begin{lemma}\label{lem:lambda-s-bound}
Let $G$ be a distance-regular graph with least eigenvalue $s$. Then
\[
    \lambda + \frac{k}{\lambda} > \frac{k}{|s|}.
\]
\end{lemma}
\begin{proof}
Let $\{u_0, u_1,\ldots,u_d\}$ be the standard sequence of polynomials for $G$
(see, e.g.,~\cite[Section 4.1B]{bcn-drg}). It is well known that $u_0(x) = 1$,
$u_1(x) = x/k$, and
\begin{equation}\label{eq:stdseq}
    c_1 u_0(x) + a_1 u_1(x) + b_1 u_2(x) = x u_1(x)
\end{equation}
(cf. eq.~(13) in~\cite[Section 4.1B]{bcn-drg}). 
Furthermore, if $\theta_i$ is the $i$th greatest eigenvalue of $G$, then the
sequence $\{u_0(\theta_i),u_1(\theta_i), \ldots, u_d(\theta_i)\}$ has exactly
$i$ sign changes~\cite[Corollary~4.1.2]{bcn-drg}. In particular, the sequence
$\{u_0(s), u_1(s),\ldots, u_d(s)\}$ is alternating, and so $u_2(s) > 0$. Hence,
from Eq.~\eqref{eq:stdseq},
\[
    \lambda -s= \frac{k}{-s} + \frac{k^2}{-s}u_2(s) > \frac{k}{-s}.
\]
So, if $\lambda \le k/|s|$, then 
$\lambda + k/\lambda > \lambda - s > k/|s|$. Thus, in any
case, $\lambda + k/\lambda > k/|s|$.   
\end{proof}

\noindent
We note that Lemma~\ref{lem:lambda-s-bound} is a slight improvement
over Lemma 3.2 of~\cite{bang-koolen} which states $\lambda+|s|>k/|s|$.
The method of proof is virtually identical.

\begin{proof}[Proof of Theorem~\ref{thm:delsarte}]
    Since $|s|\mu = o(\lambda)$, by Lemma~\ref{lem:lambda-s-bound}, we have
    \[
        k\mu < |s|\mu\left(\lambda + \frac{k}{\lambda}\right) = o(\lambda^2 + k).
    \]
    We therefore have $k\mu = o(\lambda^2)$, so by
    Theorem~\ref{thm:clique-asymptotic}, $G$ has a clique geometry $\mcC$ with
    special cliques of order $\sim \lambda$. By Lemma~\ref{lem:delsarte-bound},
    we have $\lambda \lesssim 1 + k/|s|$.  But since $\lambda \gtrsim k/|s|$ by
    Lemma~\ref{lem:lambda-s-bound}, it follows that $k/|s|$ is unbounded and
    the special cliques have order $\sim k/|s|$.   
\end{proof}

\section{Bounding the parameters of \sr\ graphs}\label{sec:srg}
A \sr\ graph with parameters $(n,k,\lambda,\mu)$ is a $k$-regular
graph on $n$ vertices such that any two adjacent vertices have exactly
$\lambda$ common neighbors, and any two distinct nonadjacent vertices have
exactly $\mu$ common neighbors. Hence, \sr\ graphs are sub-amply regular, and
indeed distance-regular if connected.  In the special case of strongly regular
graphs, we derive from our bound on $\lambda$ a bound on the nonprincipal
positive eigenvalue $r$. We compare our bounds on $\lambda$ and $r$ to those of
Spielman and of Pyber for \sr\ graphs.

Throughout this section, every \sr\ graph will have parameters
$(n,k,\lambda,\mu)$ and eigenvalues $k \ge r > s$.

\subsection{Bound on $r$}

We observe that a bound on $\lambda$ entails a corresponding bound on $r$.

We use the following
standard observations (cf.~\cite[Ch. 1.3]{bcn-drg}).
\begin{proposition}\label{prop:basic-srg}
Let $G$ be a \sr\ graph. 
\begin{itemize}
\item[(i)] $(n-k-1)\mu = k(k-\lambda-1)$
\item[(ii)] $-rs = k - \mu$
\item[(iii)] $r + s = \lambda - \mu$
\end{itemize}
\end{proposition}

\begin{corollary}\label{cor:r-bound}
Let $G$ be a \sr\ graph. Then
\[
    r < \max\left\{4\sqrt{2n}, \frac{6}{\sqrt{13}-1}\sqrt{k(\mu-1)}\right\} +
    \sqrt{k}. 
\]
\end{corollary}
\begin{proof}
By Proposition~\ref{prop:basic-srg}~(ii), we have $r < k/(-s)$, so if $-s \ge
\sqrt{k}$, the inequality is immediate. Otherwise, $-s
\le \sqrt{k}$, and so the inequality follows from
Proposition~\ref{prop:basic-srg}~(iii) and Theorem~\ref{thm:lambda}. 
\end{proof}

\subsection{Spielman's bounds for \sr\ graphs}\label{subsec:spielman}
We will state Neumaier's classification of \sr\ graphs~\cite{neumaier79}, along
with its asymptotic consequences to the parameters of \sr\ graphs, derived by
Spielman~\cite{spielman}.

A \emph{partial geometry} $\mfX = (\mcP,\mcL)$ with parameters $(R,K,\alpha)$, where
$R,K \ge 2$, is a geometric 1-design with parameters $R,K$ with the property
that for every nonincident pair $(p,\ell) \in \mcP\times \mcL$, there are
exactly $\alpha$ lines containing $p$ that intersect $\ell$.  Examples of
partial geometries include \emph{Steiner 2-designs}, which are the partial
geometries with $\alpha=K$, and \emph{transversal designs}, which are the
partial geometries with $\alpha = K-1$. The \emph{dual} of a partial geometry
$(\mcP,\mcL)$ with parameters $(R,K,\alpha)$ is the incidence structure
$(\mcL,\mcP)$. It is a partial geometry with parameters $(K,R,\alpha)$. The
\emph{line-graph} of a partial geometry is the point-graph of its dual.

Every line-graph (or point-graph) of a partial geometry is \sr, and the
geometric \sr\ graphs are point-graphs (hence line-graphs) of partial
geometries.  Other examples of \sr\ graphs include disjoint unions of cliques
of equal order and the complements of such graphs (we call these two types
\emph{trivial}); and \emph{conference graphs}, which have parameters $(n,
(n-1)/2, (n-5)/4), (n-1)/4)$.   All \sr\ graphs with a non-integral 
eigenvalue are conference graphs.

\begin{theorem}[Neumaier~{\cite{neumaier79}}]\label{thm:neumaier}
    Any \sr\ graph $G$
    is one of the following types: 
    (i) trivial; (ii) the line-graph
    of a Steiner 2-design or the line-graph of a
    transversal design; (iii) a conference graph; or
    (iv) $G$ satisfies the inequality
\begin{equation}\label{eq:claw}
r \le \max\left\{2(-s-1)(\mu + 1 + s) + s, \frac{s(s+1)(\mu+1)}{2} - 1\right\}
\end{equation}
\end{theorem}

Inequality~\eqref{eq:claw} is called the ``claw bound.'' 

The following consequences of Neumaier's classification
are implicit in Spielman's paper on testing isomorphism of \sr\ 
graphs~\cite{spielman}.

\begin{theorem}[Spielman~\cite{spielman}]\label{thm:spielman-params}
Let $G$ be a nontrivial \sr\ graph satisfying inequality~\eqref{eq:claw} (the
claw bound). Then
\begin{enumerate}
    \item[(a)] $r < k^{2/3}(\mu+1)^{1/3}$;
    \item[(b)] $\lambda < k^{2/3}(\mu+1)^{1/3}$;
\end{enumerate}
Assume furthermore that $k = o(n)$. Then
\begin{enumerate}
    \item[(c)] $\lambda = o(k)$;
    \item[(d)] $\mu \sim k^2/n$.
\end{enumerate}
\end{theorem}

Spielman explicitly states (c).  For the reader's convenience, we now give an
organized presentation of a proof of the full statement of
Theorem~\ref{thm:spielman-params}.

\begin{proof}[Proof of Theorem~\ref{thm:spielman-params}]
For any strongly regular graph, $s \le -1$ (see, e.g.,~\cite[Corollary
3.5.4]{bcn-drg}). Therefore $2(-s-1)(\mu + 1 + s)+s \le s^2(\mu+1)$, and so,
assuming the claw bound, we have
\begin{equation}\label{eq:neumaier-brief}
r \le s^2(\mu +1).
\end{equation}
Combining this with $k -\mu= -rs$ from Proposition~\ref{prop:basic-srg}~(ii)
gives
\[
r \le \left(\frac{k-\mu}{r}\right)^2(\mu+1),
\]
and hence, multiply both sides by $r^2$ and taking the $1/3$ power,
\[
r \le (k-\mu)^{2/3}(\mu+1)^{1/3}\,,
\]
proving part~(a) of the Theorem. But then combining the bound on $r$ above with
Proposition~\ref{prop:basic-srg}~(iii) we have
\[
\lambda < r + \mu < k^{2/3}(\mu+1)^{1/3}\,,
\]
proving part~(b) of the Theorem.

Now if $k = o(n)$, then $\mu = o(k)$ from Proposition~\ref{prop:basic-srg}~(i).
Then $\lambda = o(k)$ from part~(b) of the Theorem, giving part~(c). But then
part~(d) follows directly from Proposition~\ref{prop:basic-srg}~(i).
\end{proof}

\subsection{Comparison of bounds}

In this section we will compare our bounds on $\lambda$ and $r$ to those of
Spielman, as well as those of Pyber~\cite{pyber}, which we now state.

\begin{theorem}[Pyber]\label{thm:pyber-params}
Let $G$ be a nontrivial \sr\ graph. Then
\begin{enumerate}
    \item[(a)] $r < n^{1/4}k^{1/2}$;
    \item[(b)] $\lambda < n^{1/4}k^{1/2} + \mu$.
\end{enumerate}
\end{theorem}

We summarize the combination of our bound on $r$ with those of Spielman and
Pyber over the full range of possible degrees $k$. Let
\[
    g(n,k) = \min\left\{\left(\frac{k}{n}\right)^{4/3},\
                        \frac{k^{1/2}}{n^{3/4}}\,,\
        \max\left\{\left(\frac{k}{n}\right)^{3/2},\
            \left(\frac{1}{n}\right)^{1/2}\right\}\right\}.
\]

We assume $k\le (n-1)/2$ (otherwise we can take the complement of $G$). Using
the estimate $\mu = O(k^2/n)$ from Proposition~\ref{prop:basic-srg}~(i), the
following is immediate from Corollary~\ref{cor:r-bound} and
Theorems~\ref{thm:spielman-params}~(a) and \ref{thm:pyber-params}~(a).
\begin{theorem}
Let $G$ be a \sr\ graph with parameters $(n,k,\lambda,\mu)$ and eigenvalues $k
\ge r > s$ satisfying Eq.~\eqref{eq:claw} (the claw bound).
Then
\[
    \frac{r}{n} = O(g(n,k)).
\]
\end{theorem}

\begin{table}[h!]
\caption{Piecewise description of the function $g(n,k)$ giving the best known
bounds on $r/n$}\label{table:g}
\begin{centering}
\begin{tabular}{lcr}
    Value & Parameter range & Source \\
    \hline \\ [-4pt] 
    $(k/n)^{4/3}$ & $k\le n^{5/8}$ & Spielman~\cite{spielman} \\
    $n^{-1/2}$ & $n^{5/8}\le k\le n^{2/3}$ & this paper \\
    $(k/n)^{3/2}$ & $n^{2/3} \le k \le n^{3/4}$ & this paper \\
    $k^{1/2}n^{-3/4}$ & $k \ge n^{3/4}$ & Pyber~\cite{pyber} 
\end{tabular}

\end{centering}
\end{table}

Note that the function $g(n,k)$ is continuous
so up to constant factors the transition is continuous around the
boundaries of the intervals in Table~\ref{table:g}.

We now summarize the bounds on $\lambda$. Let
\[
    h(n,k) = \min\left\{\left(\frac{k}{n}\right)^{4/3},
   \ \max\left\{\frac{k^{1/2}}{n^{3/4}}\,,\ \left(\frac{k}{n}\right)^2\right\},
      \ \max\left\{\left(\frac{k}{n}\right)^{3/2},\
            \left(\frac{1}{n}\right)^{1/2}\right\}\right\}.
\]
\begin{theorem}
Let $G$ be a \sr\ graph with parameters $(n,k,\lambda,\mu)$ satisfying
Eq.~\eqref{eq:claw} (the claw bound).
Then
\[
    \frac{\lambda}{n} = O(h(n,k)).
\]
\end{theorem}

\begin{table}[h!]
\caption{Piecewise description of the function $h(n,k)$ giving the best known
bounds on $\lambda/n$}
\begin{centering}
\begin{tabular}{lcr}
    Value & Parameter range & Source \\
    \hline \\ [-4pt] 
    $(k/n)^{4/3}$ & $k\le n^{5/8}$ & Spielman~\cite{spielman} \\
    $n^{-1/2}$ & $n^{5/8}\le k\le n^{2/3}$ & this paper \\
    $(k/n)^{3/2}$ & $n^{2/3} \le k \le n^{3/4}$ & this paper \\
    $k^{1/2}n^{-3/4}$ & $n^{3/4} \le k \le n^{5/6}$ & Pyber~\cite{pyber} \\
    $(k/n)^{2}$ & $k \ge n^{5/6}$ & Pyber~\cite{pyber}
\end{tabular}

\end{centering}
\end{table}

\section{Connection to graph isomorphism testing}\label{sec:iso}

The key motivation for our main result comes from its application 
to the complexity of graph isomorphism testing (GI).  While strong
theoretical evidence suggests that this problem is not NP-complete,
the worst-case bound of $\exp(\wto(\sqrt{n}))$, established three
decades ago~\cite{bl-canonical,zemlyachenko,bkantorluks},
continues to be unchallenged.  

Strongly regular graphs have long been recognized as a difficult
although probably not complete class for GI; there has been
slightly more progress on the complexity of testing their
isomorphism.  The first bound for \sr\ graphs was
$\exp(\wto(n^{1/2}))$~\cite{bab-srg} (1980), followed by
$\exp(\wto(n^{1/3}))$~\cite{spielman} (Spielman, 1996) and
$\exp(\wto(n^{1/5}))$~\cite{bcstw-sr} (2013).  The two main components
of the recent result are an $\exp(n^{O(\mu +\log n)})$ bound and an
$\exp(n^{\wto(1+\lambda/\mu)})$ bound.  While under Neumaier's claw
bound, the value of $\mu$ is asymptotically determined by $n$ 
and $k$ ($\mu\sim k^2/n$, see Theorem~\ref{thm:spielman-params}~(d)), 
the value of $\lambda$ can vary
widely, thus the significance of an improved bound on $\lambda$
that contributed to reducing the exponent of the exponent to
$1/5$.

\section{Conclusion and open problems}\label{sec:conclusion}

We have derived a new bound on the parameter $\lambda$ of sub-amply regular
graphs, and hence for distance-regular graphs.  In the particular case of \sr\
graphs, the improved bound contributed to the improved complexity estimate
for testing isomorphism of \sr\ graphs~\cite{bcstw-sr}. Our proof relies on 
Metsch's clique geometry when $k\mu = o(\lambda^2)$.

Examples of this clique structure arise in geometric \sr\ graphs, in
particular in point-graphs of partial geometries, including
Steiner designs and their duals.

We are not aware of infinite families of sub-amply regular graphs 
satisfying $k\mu = o(\lambda^2)$ which are not in fact point-graphs of
geometric 1-designs.  If such families do not exist, this would considerably
strengthen the conclusion of Theorem~\ref{thm:clique-asymptotic}.

In fact, we are not aware of even a single non-geometric sub-amply regular
graph satisfying inequality~\eqref{eq:main}.

We note that if any examples of non-geometric strongly regular graphs
satisfying inequality~\eqref{eq:main} exist, they will be rather large. 
No example has fewer than 1500 vertices; this was verified by checking all
feasible parameters of strongly regular graphs in the table compiled by Andries
Brouwer~\cite{srg-table}.

\bibliographystyle{amsplain}
\bibliography{delsarte}
\end{document}